\documentclass[11pt]{article}

\usepackage[T1]{fontenc}
\usepackage{lmodern}
\usepackage{amsmath,amssymb,amsthm,mathtools}
\usepackage{enumitem}
\usepackage{microtype}
\usepackage[hidelinks]{hyperref}
\usepackage{xcolor,soul}
\usepackage{tikz}
\usetikzlibrary{graphs,positioning,arrows.meta}

\usepackage[margin=1.08in]{geometry}

\newtheorem{theorem}{Theorem}[section]
\newtheorem{lemma}[theorem]{Lemma}
\newtheorem{conjecture}[theorem]{Conjecture}

\newtheorem{claim}[theorem]{Claim}
\newtheorem{corollary}[theorem]{Corollary}
\theoremstyle{definition}

\theoremstyle{remark}

\newcommand{\rn}{\mathrm{rn}}
\newcommand{\col}{\varphi}

\date{}

\title{Rich-neighbor \(5\)-edge-colorings of subcubic graphs whose
\(3\)-vertices induce a matching}

\author{Jiangxu Kong$^1$,
 \ \ Zhuming Lin$^1$, \ \ Yan Wang$^{2,}$\thanks{\footnotesize  Corresponding author. Email: wangyan@mnnu.edu.cn‌.} \\
\normalsize 1. School of Mathematics, Hangzhou Normal University, Hangzhou, 311121, China\\
\normalsize 2. School of mathematics and statistics, Fujian Key Laboratory of \\ Granular Computing and Applications, \\ Minnan Normal University, Zhangzhou 363000, China
}

\begin{document}
\maketitle

\begin{abstract}
Let $\varphi$ be a proper edge-coloring of a graph $G$. An edge is called rich if all edges adjacent to it receive pairwise distinct colors. The coloring $\varphi$ is a rich-neighbor edge-coloring if every non-isolated edge is adjacent to a rich edge. Petruševski and Škrekovski introduced this notion as a relaxation of strong edge-coloring and conjectured that every connected subcubic graph other than $K_4$ admits a rich-neighbor $5$-edge-coloring.

We establish the conjectured bound for every subcubic graph $G$ satisfying $\Delta(G[V_3])\le 1,$ where $V_3$ is the set of vertices of degree $3$ in $G$. Equivalently, the $3$-vertices of $G$ induce a matching together with isolated vertices. This extends the result of Chen, Tan and Zhou for subcubic graphs with no adjacent $3$-vertices.
\end{abstract}

\begin{center}
\begin{minipage}{140mm}
\end{minipage}
\end{center}

\maketitle

\section{Introduction}

All graphs considered in this paper are finite, undirected, and simple, unless a multigraph is explicitly introduced. For a positive integer $k$, let $[k]=\{1,\ldots,k\}$. A vertex of degree one is called a \emph{leaf}, and an edge incident with a leaf is a \emph{pendant edge}. A \emph{bridge} is an edge whose deletion increases the number of connected components. An edge is \emph{isolated} if it forms a
component isomorphic to $K_2$. Two edges are \emph{adjacent} if they share an endpoint. For an edge $e$ of a graph $G$, define its open edge-neighborhood by $N_G(e)=\{f\in E(G): f\text{ is adjacent to }e\}.$

An edge-coloring $\col\colon E(G)\to[k]$ is \emph{proper} if adjacent edges receive distinct colors. Under a proper edge-coloring $\col$, an edge $e$ is \emph{rich} if $\col$ is injective on $N_G(e)$. The
coloring $\col$ is a \emph{rich-neighbor edge-coloring} if every non-isolated edge has a rich edge in its open edge-neighborhood. The minimum number of colors in a rich-neighbor edge-coloring of $G$ is
denoted by $\chi'_{\rn}(G)$.

Rich-neighbor edge-coloring is closely related to strong edge-coloring. A strong edge-coloring of a graph $G$ is a proper edge-coloring in which every color class is an induced matching.
%Equivalently, distinct edges $e$ and $f$ receive different colors whenever $d_{L(G)}(e,f)\le 2,$ where $L(G)$ is the line graph of $G$. 
Strong edge-coloring was
introduced by Fouquet and Jolivet~\cite{FJ83}, and its study has been
motivated in particular by the conjecture of Erd\H{o}s and
Ne\v{s}et\v{r}il~\cite{EN89,FGST90}. The minimum number of colors in a
strong edge-coloring of $G$ is the \emph{strong chromatic index},
denoted by $\chi'_s(G)$. Every edge is rich under a strong edge-coloring. Consequently,
$\chi'_{\rn}(G)\le \chi'_s(G).$ For subcubic graphs, Andersen~\cite{Andersen92} and, independently, Hor\'ak, Qing, and Trotter~\cite{HQT93} proved the sharp bound $\chi'_s(G)\le 10.$
The rich-neighbor condition is substantially weaker: instead of
requiring every edge to be rich, it requires only that every
non-isolated edge be adjacent to a rich edge.

There is also a connection with normal edge-colorings of cubic graphs.
Let $\col$ be a proper edge-coloring of a cubic graph, and let $uv$ be
an edge. The edge $uv$ is called \emph{poor} if the edges incident
with $u$ or $v$ use exactly three colors, and \emph{rich} if they use
exactly five colors. Thus the latter notion agrees with the definition
of a rich edge given above. A proper edge-coloring is \emph{normal} if
every edge is either poor or rich. Normal edge-colorings were
introduced by Jaeger~\cite{Jae85,Jae88} in connection with the Petersen
coloring conjecture. The assertion that every bridgeless cubic graph
admits a normal $5$-edge-coloring is equivalent to that conjecture.
Mazzuoccolo and Mkrtchyan~\cite{MM20} subsequently proved that every
simple cubic graph admits a normal $7$-edge-coloring and that this
bound is best possible.

Petru\v{s}evski and \v{S}krekovski~\cite{PS} introduced rich-neighbor
edge-coloring as a relaxation of strong edge-coloring. They proved
that every subcubic graph admits a rich-neighbor $7$-edge-coloring and
proposed the following conjecture.

\begin{conjecture}\label{conj:PS}
Every connected subcubic graph other than $K_4$ admits a rich-neighbor
$5$-edge-coloring.
\end{conjecture}

Chen, Tan, and Zhou~\cite{CTZ} confirmed
Conjecture~\ref{conj:PS} for subcubic graphs with no adjacent
$3$-vertices. More recently, Chen, Nan, and Zhou~\cite{CNZ} proved two
rich-neighbor $6$-edge-coloring results. They showed that
$\chi'_{\rn}(G)\le 6$ if $G$ is a subcubic graph satisfying
$\operatorname{mad}(G)<36/13,$
where $\operatorname{mad}(G)$ denotes the maximum average degree of
$G$, and also if $G$ is a claw-free subcubic graph.

For a graph $G$, let $V_3(G)=\{v\in V(G):d_G(v)=3\},$
and write $V_3$ for $V_3(G)$ when the host graph is clear. The
condition
$\Delta(G[V_3])\le 1$
is equivalent to requiring every component of $G[V_3]$ to be
isomorphic to either $K_1$ or $K_2$. In particular, adjacent pairs of
$3$-vertices are allowed, provided that they form a matching. Our main
result extends the theorem of Chen, Tan, and Zhou from the case in
which $G[V_3]$ is edgeless to the case in which it has maximum degree
at most one.

\begin{theorem}\label{thm:main}
Let $G$ be a subcubic graph satisfying
$\Delta\bigl(G[V_3]\bigr)\le 1.$
Then
$\chi'_{\rn}(G)\le 5.$
\end{theorem}

Equivalently, Theorem~\ref{thm:main} applies whenever every
$3$-vertex of $G$ is adjacent to at most one other $3$-vertex. The
exceptional graph $K_4$ from Conjecture~\ref{conj:PS} is excluded by
the hypothesis, since $K_4[V_3]=K_4$.

The bound in Theorem~\ref{thm:main} is best possible. Let
$D=K_4-cd$, where $V(D)=\{a,b,c,d\}$; thus $a$ and $b$ are the
$3$-vertices of $D$. Then
$\Delta\bigl(D[V_3]\bigr)=1.$
Suppose that $D$ admits a rich-neighbor edge-coloring $\col$ using at
most four colors. The edge $ab$ is adjacent to every other edge of
$D$. Since these four edges avoid the color $\col(ab)$, they use at
most three colors, and hence $ab$ is not rich. Applying the
rich-neighbor condition to $ac$, we find that either $ad$ or $bc$ is
rich. In either case, the edges $ab$, $ac$, and $bd$ receive pairwise
distinct colors, so
$\col(ac)\ne\col(bd).$
Similarly, applying the rich-neighbor condition to $ad$ gives
$\col(ad)\ne\col(bc).$
These are the only two pairs of nonadjacent edges in $D$. Properness
therefore implies that all five edges of $D$ receive distinct colors,
a contradiction. Conversely, assigning distinct colors to the five
edges gives a rich-neighbor $5$-edge-coloring. Hence
$\chi'_{\rn}(D)=5.$

The main obstacle in proving Theorem~\ref{thm:main} is that the
hypothesis is not preserved by the usual operation of suppressing a
$2$-vertex. Indeed, replacing a path $uxv$, where $d_G(x)=2$, by the
edge $uv$ may create an additional adjacency between two
$3$-vertices. If each of $u$ and $v$ already has a neighbor in
$V_3$, then the resulting graph no longer satisfies
$\Delta(G[V_3])\le 1$. A direct reduction at a single $2$-vertex is
therefore unavailable. We instead treat all paths between
$3$-vertices simultaneously and coordinate the colors at their
endpoints.

The proof uses a thread decomposition with respect to
$B=V_3(G)$. Maximal paths with ends in $B$ and internal vertices
outside $B$ are represented by the edges of an auxiliary multigraph
on $B$. The threads of length one form a matching by the hypothesis
$\Delta(G[V_3])\le 1$. After separating these threads, we choose an
minimal edge cover in the remaining auxiliary multigraph.
The corresponding selected threads are colored by a fixed periodic
pattern that produces the required rich witnesses and leaves
controlled lists of available colors at their endpoints. The
unselected threads and pendant edges are then represented by an
auxiliary graph of maximum degree at most two and colored through a
list edge-coloring argument. This global coordination is what permits
adjacent pairs of $3$-vertices while retaining the conjectured bound
of five colors.

\section{Preliminaries}

Isolated vertices may be deleted. A component isomorphic to $K_2$
may likewise be removed and restored after the remaining graph has
been colored, since its unique edge has no neighbor and may be
assigned any color from the same palette. As distinct components can
be colored independently, it suffices to prove
Theorem~\ref{thm:main} for connected graphs of order at least three.

The hypothesis of Theorem~\ref{thm:main} is hereditary.

\begin{lemma}\label{lem:hereditary}
Let $G$ be a subcubic graph satisfying
$\Delta\bigl(G[V_3(G)]\bigr)\le 1.$
Then every subgraph $H$ of $G$ satisfies
$\Delta\bigl(H[V_3(H)]\bigr)\le 1.$
\end{lemma}

\begin{proof}
If $v\in V_3(H)$, then
$3=d_H(v)\le d_G(v)\le 3,$
so $v\in V_3(G)$. Hence
$V_3(H)\subseteq V_3(G)$ and $H[V_3(H)]\subseteq G[V_3(G)].$
The assertion follows.
\end{proof}

We shall use the following standard strong edge-coloring bound.

\begin{lemma}[{\cite{CCHHL,FGST90}}]\label{lem:trees-cycles}
Every subcubic tree or cycle admits a strong $5$-edge-coloring.
\end{lemma}

The next lemma allows the colors on the two end edges of a sufficiently
long path to be prescribed.

\begin{lemma}\label{lem:path-sequence}
Let $P=x_0x_1\cdots x_\ell$ be a path with $\ell\ge 4$ edges, and write $e_i=x_{i-1}x_i$ for $i\in[\ell]$. For any $a,b\in[5]$, the path $P$ has a strong $5$-edge-coloring $\psi$ satisfying
$\psi(e_1)=a$ and  $\psi(e_\ell)=b.$
\end{lemma}

\begin{proof}
Set $c_1=a$ and $c_\ell=b$. First choose
$c_2,\ldots,c_{\ell-3}$ successively, whenever this sequence is
nonempty, so that each $c_i$ avoids the already defined colors among
$c_{i-1}$ and $c_{i-2}$. Next choose $c_{\ell-2}$ and
$c_{\ell-1}$ successively, requiring each $c_i$ to avoid the already
defined colors among
$c_{i-1}, c_{i-2}, b.$
At most three colors are forbidden at every step, so all choices can
be made from $[5]$. The resulting sequence satisfies
$c_i\ne c_j$ whenever $0<|i-j|\le 2.$
Thus $\psi(e_i)=c_i$ defines the required strong edge-coloring.
\end{proof}

\section{Proof of Theorem~\ref{thm:main}}

Suppose, to the contrary, that Theorem~\ref{thm:main} is false.
Since distinct components may be colored independently, there exists
a connected counterexample. Let $G$ be such a counterexample with the
minimum number of edges. By Lemma~\ref{lem:trees-cycles}, the graph
$G$ is neither a tree nor a cycle. Hence $G$ contains a cycle and,
since it is connected and subcubic, it has a $3$-vertex.

We first eliminate nonpendant bridges.

\begin{lemma}\label{lem:bridge}
Every bridge of $G$ is a pendant edge.
\end{lemma}

\begin{proof}
Suppose that $e=uv$ is a nonpendant bridge. Let $C_u$ and $C_v$ be
the components of $G-e$ containing $u$ and $v$, respectively. Since
$e$ is not pendant, both components contain an edge. Set
\[
  G_u=G\bigl[V(C_u)\cup\{v\}\bigr],
  \qquad
  G_v=G\bigl[V(C_v)\cup\{u\}\bigr].
\]
Both $G_u$ and $G_v$ are proper subgraphs of $G$ with fewer edges.
By Lemma~\ref{lem:hereditary} and the minimality of $G$, they admit
rich-neighbor $5$-edge-colorings $\col_u$ and $\col_v$. After
permuting the colors in $\col_v$, we may assume that
$\col_u(e)=\col_v(e)=\alpha.$

For a vertex $x$ of a graph $H$, let $E_H(x)$ denote the set of edges
of $H$ incident with $x$. Put
\[
  A_u=\col_u\bigl(E_{G_u}(u)\setminus\{e\}\bigr),
  \qquad
  A_v=\col_v\bigl(E_{G_v}(v)\setminus\{e\}\bigr).
\]
Properness and subcubicity give
$A_u,A_v\subseteq[5]\setminus\{\alpha\},$ $|A_u|,|A_v|\le 2.$
Since $[5]\setminus\{\alpha\}$ has four elements, there is a
permutation $\sigma$ of $[5]$ fixing $\alpha$ such that
$A_u\cap\sigma(A_v)=\varnothing.$
Replace $\col_v$ by $\sigma\circ\col_v$. The two colorings now agree
on $e$ and combine into a proper $5$-edge-coloring $\col$ of $G$.

For every $f\in E(G_u)\setminus\{e\}$,
$N_G(f)=N_{G_u}(f),$
so $f$ is rich under $\col$ if and only if it is rich under $\col_u$.
The analogous statement holds on the $v$-side, since permuting colors
preserves richness. Moreover,
$\col\bigl(N_G(e)\bigr)=A_u\cup\sigma(A_v).$
The colors in this union are pairwise distinct, and hence $e$ is
rich under $\col$.

Let $f\ne e$ be a non-isolated edge of $G$. In its side coloring,
$f$ has a rich neighbor. If that neighbor is not $e$, it remains rich
under $\col$; if it is $e$, then it is rich by the preceding
paragraph. Thus every non-isolated edge other than $e$ has a rich
neighbor.

Finally, since $C_u$ contains an edge, $e$ is non-isolated in $G_u$.
It therefore has a rich neighbor
$g\in E(G_u)\setminus\{e\}$ under $\col_u$. The edge $g$ remains rich
under $\col$, so $e$ also has a rich neighbor. Consequently, $\col$
is a rich-neighbor $5$-edge-coloring of $G$, a contradiction.
\end{proof}

As an immediate consequence, leaves can occur only in a restricted form.

\begin{corollary}\label{cor:leaf-support}
If $x$ is a leaf of $G$ with neighbor $v$, then $d_G(v)=3$, and $x$ is the unique leaf neighbor of $v$.
\end{corollary}

\begin{proof}
Since $G$ is connected and has more than one edge, $d_G(v)\ge2$. Suppose that either $d_G(v)=2$ or $v$ has another leaf neighbor. Then $v$ has a neighbor $w$ such that every neighbor of $v$ other
than $w$ is a leaf. Consequently, $vw$ lies on no cycle and is therefore a bridge.

If $w$ is a leaf, then every neighbor of $v$ is a leaf, so the connected graph $G$ is a star and hence a tree, a contradiction. Otherwise, $vw$ is a nonpendant bridge, contradicting Lemma~\ref{lem:bridge}. Therefore $d_G(v)=3$, and $x$ is the unique leaf neighbor of $v$.
\end{proof}

Let $V_1(G)$ denote the set of leaves of $G$, and put $\hat G=G-V_1(G);$ see Figure~\ref{fig:G-and-Ghat}, where the filled vertices are $3$-vertices of $G$ and the unfilled vertices are $1$- or $2$-vertices. A leaf cannot be an internal vertex of a path. Hence every path in $G$ joining two vertices of $\hat G$ is contained entirely in $\hat G$.

\begin{figure}[htbp]
\centering
\begin{tikzpicture}[x=0.75cm,y=0.75cm]

% ============================================================
% Left: G
% ============================================================

% Coordinates for G
\coordinate (g1)  at (3.0,3.35);
\coordinate (g2)  at (0.8,2.0);
\coordinate (g3)  at (5.2,2.0);
\coordinate (g4)  at (1.0,0.0);
\coordinate (g5)  at (5.0,0.0);

\coordinate (g6)  at (1.90,2.70);
\coordinate (g7)  at (4.10,2.70);
\coordinate (g8)  at (3.00,0.00);
\coordinate (g9)  at (5.10,1.00);
\coordinate (g10) at (2.35,2.30);
\coordinate (g11) at (1.65,1.15);
\coordinate (g12) at (1.35,3.75);
\coordinate (g13) at (3.00,4.15);
\coordinate (g14) at (4.65,3.75);
\coordinate (g15) at (6.05,0.00);

% Edges of G
\draw[black,line width=0.9pt] (g1)--(g6)--(g2);
\draw[black,line width=0.9pt] (g1)--(g7)--(g3);
\draw[black,line width=0.9pt] (g4)--(g8)--(g5);
\draw[black,line width=0.9pt] (g2)--(g4);
\draw[black,line width=0.9pt] (g3)--(g9)--(g5);
\draw[black,line width=0.9pt] (g1)--(g10)--(g11)--(g4);
\draw[black,line width=0.9pt] (g2)--(g12)--(g13)--(g14)--(g3);
\draw[black,line width=0.9pt] (g5)--(g15);

% Filled vertices in B
\foreach \p in {g1,g2,g3,g4,g5}{
  \fill[black] (\p) circle (2.2pt);
}

% Open vertices outside B
\foreach \p in {g6,g7,g8,g9,g10,g11,g12,g13,g14,g15}{
  \draw[black,fill=white,line width=0.8pt] (\p) circle (2.2pt);
}

% Labels of G
\node[font=\scriptsize,fill=white,inner sep=1pt] at (3.0,3.72) {$v_1$};
\node[font=\scriptsize,fill=white,inner sep=1pt] at (0.32,2.15) {$v_2$};
\node[font=\scriptsize,fill=white,inner sep=1pt] at (5.68,2.15) {$v_3$};
\node[font=\scriptsize,fill=white,inner sep=1pt] at (0.52,-0.30) {$v_4$};
\node[font=\scriptsize,fill=white,inner sep=1pt] at (5.12,-0.38) {$v_5$};

\node[font=\scriptsize,fill=white,inner sep=1pt] at (1.65,3.05) {$v_6$};
\node[font=\scriptsize,fill=white,inner sep=1pt] at (4.35,3.05) {$v_7$};
\node[font=\scriptsize,fill=white,inner sep=1pt] at (3.00,-0.38) {$v_8$};
\node[font=\scriptsize,fill=white,inner sep=1pt] at (5.52,1.10) {$v_9$};
\node[font=\scriptsize,fill=white,inner sep=1pt] at (2.78,2.14) {$v_{10}$};
\node[font=\scriptsize,fill=white,inner sep=1pt] at (2.2,1.1) {$v_{11}$};
\node[font=\scriptsize,fill=white,inner sep=1pt] at (1.12,4.10) {$v_{12}$};
\node[font=\scriptsize,fill=white,inner sep=1pt] at (3.00,4.55) {$v_{13}$};
\node[font=\scriptsize,fill=white,inner sep=1pt] at (4.88,4.10) {$v_{14}$};
\node[font=\scriptsize,fill=white,inner sep=1pt] at (6.50,-0.05) {$v_{15}$};

% Label below the left graph
\node[font=\small] at (3.05,-0.95) {$G$};

% Arrow and operation label
\draw[->,black,line width=0.9pt] (6.85,2.05)--(7.75,2.05);
\node[font=\small,fill=white,inner sep=1pt] at (7.30,2.6) {$\widehat{G}=G-V_1$};

% ============================================================
% Right: Ghat = G - V_1
% ============================================================

\begin{scope}[shift={(8.2,0)}]

% Coordinates for Ghat
\coordinate (h1)  at (3.0,3.35);
\coordinate (h2)  at (0.8,2.0);
\coordinate (h3)  at (5.2,2.0);
\coordinate (h4)  at (1.0,0.0);
\coordinate (h5)  at (5.0,0.0);

\coordinate (h6)  at (1.90,2.70);
\coordinate (h7)  at (4.10,2.70);
\coordinate (h8)  at (3.00,0.00);
\coordinate (h9)  at (5.10,1.00);
\coordinate (h10) at (2.35,2.30);
\coordinate (h11) at (1.65,1.15);
\coordinate (h12) at (1.35,3.75);
\coordinate (h13) at (3.00,4.15);
\coordinate (h14) at (4.65,3.75);

% Edges of Ghat
\draw[black,line width=0.9pt] (h1)--(h6)--(h2);
\draw[black,line width=0.9pt] (h1)--(h7)--(h3);
\draw[black,line width=0.9pt] (h4)--(h8)--(h5);
\draw[black,line width=0.9pt] (h2)--(h4);
\draw[black,line width=0.9pt] (h3)--(h9)--(h5);
\draw[black,line width=0.9pt] (h1)--(h10)--(h11)--(h4);
\draw[black,line width=0.9pt] (h2)--(h12)--(h13)--(h14)--(h3);

% Filled vertices in B
\foreach \p in {h1,h2,h3,h4,h5}{
  \fill[black] (\p) circle (2.2pt);
}

% Open vertices outside B
\foreach \p in {h6,h7,h8,h9,h10,h11,h12,h13,h14}{
  \draw[black,fill=white,line width=0.8pt] (\p) circle (2.2pt);
}

% Labels of Ghat
\node[font=\scriptsize,fill=white,inner sep=1pt] at (3.18,3.72) {$v_1$};
\node[font=\scriptsize,fill=white,inner sep=1pt] at (0.32,2.15) {$v_2$};
\node[font=\scriptsize,fill=white,inner sep=1pt] at (5.68,2.15) {$v_3$};
\node[font=\scriptsize,fill=white,inner sep=1pt] at (0.52,-0.30) {$v_4$};
\node[font=\scriptsize,fill=white,inner sep=1pt] at (5.12,-0.38) {$v_5$};

\node[font=\scriptsize,fill=white,inner sep=1pt] at (1.65,3.05) {$v_6$};
\node[font=\scriptsize,fill=white,inner sep=1pt] at (4.35,3.05) {$v_7$};
\node[font=\scriptsize,fill=white,inner sep=1pt] at (3.00,-0.38) {$v_8$};
\node[font=\scriptsize,fill=white,inner sep=1pt] at (5.52,1.10) {$v_9$};
\node[font=\scriptsize,fill=white,inner sep=1pt] at (2.78,2.14) {$v_{10}$};
\node[font=\scriptsize,fill=white,inner sep=1pt] at (2.2,1.1) {$v_{11}$};
\node[font=\scriptsize,fill=white,inner sep=1pt] at (1.12,4.10) {$v_{12}$};
\node[font=\scriptsize,fill=white,inner sep=1pt] at (3.00,4.55) {$v_{13}$};
\node[font=\scriptsize,fill=white,inner sep=1pt] at (4.88,4.10) {$v_{14}$};

% Label below the right graph
\node[font=\small] at (3.05,-0.95) {$\widehat{G}=G-V_1$};

\end{scope}

\end{tikzpicture}

\caption{The original graph $G$ and the graph $\widehat{G}=G-V_1$. }
\label{fig:G-and-Ghat}
\end{figure}
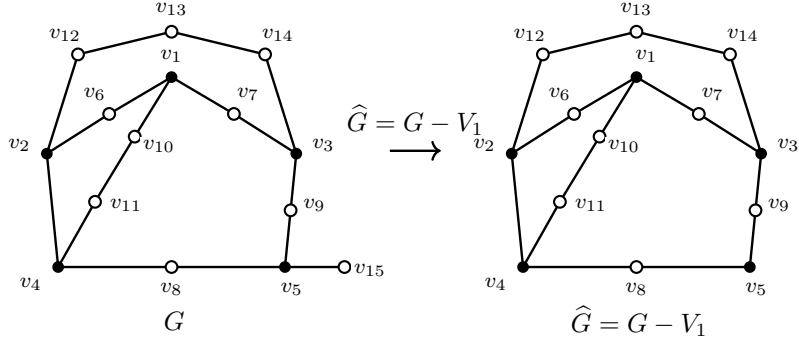

\begin{claim}\label{cl:K-core}
The graph \(\hat G\) is connected and bridgeless, with
\(\delta(\hat G)\ge2\). If \(|V(\hat G)|\ge3\), then \(\hat G\) is
\(2\)-connected.
\end{claim}

\begin{proof}
Every vertex of \(\hat G\) has degree two or three in \(G\). By
Corollary~\ref{cor:leaf-support}, a \(2\)-vertex has no leaf neighbor
and a \(3\)-vertex has at most one. Hence
\(\delta(\hat G)\ge2\). Moreover, no leaf can be an internal vertex
of a path, so the connectivity of \(G\) implies that of \(\hat G\).

If \(ab\) were a bridge of \(\hat G\), then any \(a\)-\(b\) path in
\(G-ab\) would lie in \(\hat G\), a contradiction. Thus \(ab\) would
also be a bridge of \(G\). Since
\(d_G(a),d_G(b)\ge2\), this bridge would not be pendant, contrary to
Lemma~\ref{lem:bridge}. Therefore \(\hat G\) is bridgeless.

Finally, suppose that \(|V(\hat G)|\ge3\) and \(v\) is a cut vertex of
\(\hat G\). Each component of \(\hat G-v\) is joined to \(v\). Since
\(d_{\hat G}(v)\le3\) and \(\hat G-v\) has at least two components,
one of them is joined to \(v\) by exactly one edge. This edge is a
bridge of \(\hat G\), a contradiction. Hence \(\hat G\) is
\(2\)-connected.
\end{proof}

\begin{claim}\label{cl:B-at-least-two}
$|V_3(G)|\ge2$.
\end{claim}

\begin{proof}
If $V_3(G)=\varnothing$, then the connected graph $G$ has maximum degree at most two. Since $G$ contains a cycle, it must itself be a cycle, contradicting Lemma~\ref{lem:trees-cycles}.

Suppose that $V_3(G)=\{v\}$. Every vertex of $\hat G$ other than $v$ has degree two in $\hat G$. By the handshaking lemma, $d_{\hat G}(v)$ is even. Since $2\le d_{\hat G}(v)\le3,$ we have $d_{\hat G}(v)=2$. Thus $\hat G$ is a connected $2$-regular graph and hence a cycle. Corollary~\ref{cor:leaf-support} now implies that $G$ is obtained from $\hat G$ by attaching exactly one pendant edge $g$ at $v$.

Let \(\varphi\) be a strong \(5\)-edge-coloring of \(\hat G\). Let
\(e,e'\) be the cycle edges incident with \(v\), and let \(f,f'\) be
their other cycle neighbors, where \(f=f'\) if \(\hat G\) is a
triangle. Choose
\(
\beta\in[5]\setminus
 \{\varphi(e),\varphi(e'),\varphi(f),\varphi(f')\}
\)
and set \(\varphi(g)=\beta\). The extension is proper, and the choice
of \(\beta\) ensures that \(e\) and \(e'\) remain rich. The
neighborhood of every other cycle edge is unchanged, so all cycle
edges remain rich. Hence every edge of \(G\) has a rich neighbor, a
contradiction.

\end{proof}

Put $B=V_3(G)$. A \emph{$B$-thread} is a maximal path in $\hat G$ with both endpoints in $B$ and all internal vertices outside $B$. The length of a thread is its number of edges.

Every vertex of $\hat G-B$ has degree two in $\hat G$. Hence every
component of $\hat G-B$ is either a path, possibly consisting of a
single vertex, or a cycle. A cycle component would have no neighbor
in $B$, contrary to the connectivity of $\hat G$. Thus every
component of $\hat G-B$ is a path.

Each such component is joined to $B$ by exactly two edges. These two
edges have distinct endpoints in $B$. Indeed, if both ended at some
$v\in B$, then, since $|B|\ge2$, deleting $v$ would separate that
component from the remaining vertices of $\hat G$, contradicting
Claim~\ref{cl:K-core}. Consequently, each component of $\hat G-B$,
together with its two attachment edges, forms a $B$-thread with
distinct endpoints. In addition, every edge of $\hat G[B]$ is a
$B$-thread of length one. Therefore the $B$-threads partition
$E(\hat G)$.

Compress every $B$-thread to a single edge joining its endpoints.
The resulting graph is a loopless multigraph $G^*$ with
$V(G^*)=B$ and $d_{G^*}(v)=d_{\hat G}(v)\in\{2,3\}$ for every $v\in B;$
see Figure~\ref{fig:Gstar-and-Gprime}. Parallel edges may occur
because distinct threads may have the same pair of endpoints.

Let $M$ be the set of edges of $G^*$ corresponding to threads of
length one. These are precisely the edges of $G$ joining two
$3$-vertices. Since $\Delta(G[B])\le1$,
\begin{equation}\label{eq:M-matching1}
  M\text{ is a matching in }G^*.
\end{equation}
Set $G'=G^*-M;$
see Figure~\ref{fig:Gstar-and-Gprime}. For every $v\in B$,
$d_{G'}(v)=d_{G^*}(v)-d_M(v)\ge2-1=1.$
Thus $G'$ has no isolated vertex.

Recall that an \emph{edge cover} of a graph $H$ is a set $F\subseteq E(H)$ incident with every vertex of $H$. It is \emph{minimal} if no proper subset of $F$ is an edge cover.
Since $G'$ has no isolated vertex, it has an edge cover. Fix an
minimal edge cover $F\subseteq E(G')$, and let $G'_F$ be the
spanning submultigraph of $G'$ with edge set $F$; see
Figure~\ref{fig:Gprime-and-GprimeF}.

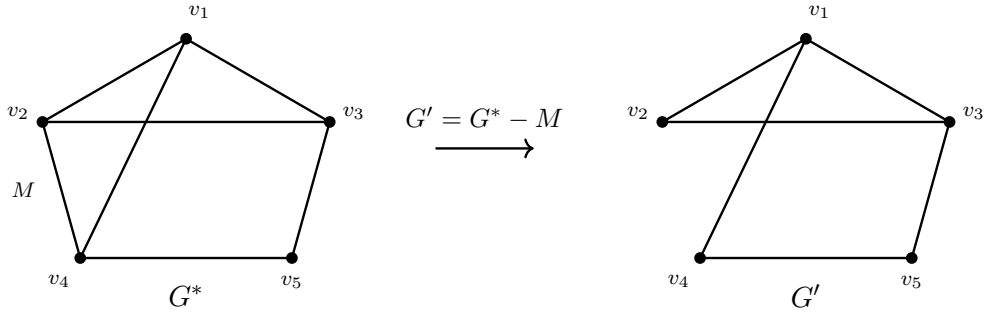
\begin{figure}[htbp]
\centering
\begin{tikzpicture}[x=1cm,y=1cm]

% ============================================================
% Left: G^*
% ============================================================

\node[font=\bfseries] at (2.7,-0.3) {$G^{*}$};

% Coordinates for G^*
\coordinate (a1) at (2.7,3.1);
\coordinate (a2) at (0.8,2.0);
\coordinate (a3) at (4.6,2.0);
\coordinate (a4) at (1.3,0.2);
\coordinate (a5) at (4.1,0.2);

% Edges of G^*
\draw[black,line width=0.9pt] (a1)--(a2);
\draw[black,line width=0.9pt] (a1)--(a3);
\draw[black,line width=0.9pt] (a4)--(a5);
\draw[black,line width=0.9pt] (a2)--(a4);   % edge in M
\draw[black,line width=0.9pt] (a3)--(a5);
\draw[black,line width=0.9pt] (a1)--(a4);
\draw[black,line width=0.9pt] (a2)--(a3);

% Vertices of G^*
\foreach \p in {a1,a2,a3,a4,a5}{
  \fill[black] (\p) circle (2.2pt);
}

% Labels of vertices
\node[font=\scriptsize,fill=white,inner sep=1pt] at (2.88,3.45) {$v_1$};
\node[font=\scriptsize,fill=white,inner sep=1pt] at (0.48,2.12) {$v_2$};
\node[font=\scriptsize,fill=white,inner sep=1pt] at (4.92,2.12) {$v_3$};
\node[font=\scriptsize,fill=white,inner sep=1pt] at (1.02,-0.10) {$v_4$};
\node[font=\scriptsize,fill=white,inner sep=1pt] at (4.10,-0.12) {$v_5$};

% Mark the edge in M
\node[font=\scriptsize,fill=white,inner sep=1pt] at (0.55,1.10) {$M$};

% Under-text
%\node[font=\small] at (3.2,-0.3) {$M=\{v_2v_4\}$};

% Arrow between the two graphs
\draw[->,black,line width=0.9pt] (6.0,1.65)--(7.3,1.65);
\node[font=\small,fill=white,inner sep=1pt] at (6.65,2.05) {$G' = G^{*}-M$};

% ============================================================
% Right: G'
% ============================================================

\begin{scope}[shift={(8.2,0)}]

\node[font=\bfseries] at (2.7,-0.3) {$G'$};

% Coordinates for G'
\coordinate (b1) at (2.7,3.1);
\coordinate (b2) at (0.8,2.0);
\coordinate (b3) at (4.6,2.0);
\coordinate (b4) at (1.3,0.2);
\coordinate (b5) at (4.1,0.2);

% Edges of G'
\draw[black,line width=0.9pt] (b1)--(b2);
\draw[black,line width=0.9pt] (b1)--(b3);
\draw[black,line width=0.9pt] (b4)--(b5);
\draw[black,line width=0.9pt] (b3)--(b5);
\draw[black,line width=0.9pt] (b1)--(b4);
\draw[black,line width=0.9pt] (b2)--(b3);

% Vertices of G'
\foreach \p in {b1,b2,b3,b4,b5}{
  \fill[black] (\p) circle (2.2pt);
}

% Labels of vertices
\node[font=\scriptsize,fill=white,inner sep=1pt] at (2.88,3.45) {$v_1$};
\node[font=\scriptsize,fill=white,inner sep=1pt] at (0.48,2.12) {$v_2$};
\node[font=\scriptsize,fill=white,inner sep=1pt] at (4.92,2.12) {$v_3$};
\node[font=\scriptsize,fill=white,inner sep=1pt] at (1.02,-0.10) {$v_4$};
\node[font=\scriptsize,fill=white,inner sep=1pt] at (4.10,-0.12) {$v_5$};

\end{scope}

\end{tikzpicture}
\caption{The multigraph $G^{*}$ and the graph $G'=G^{*}-M$.}
\label{fig:Gstar-and-Gprime}
\end{figure}

\begin{lemma}\label{lem:star-forest}
Every component of $G'_F$ is a star, and no two edges of $F$ are
parallel.
\end{lemma}

\begin{proof}
By the minimality of $F$, every edge $xy\in F$ has an
endpoint of degree one in $G'_F$; otherwise, both $x$ and $y$ would
remain covered after deleting $xy$, and $F-\{xy\}$ would still be an
edge cover of $G'$. In particular, $F$ contains no parallel edges,
since the endpoints of two parallel edges would both have degree at
least two in $G'_F$.

Let $C$ be a component of $G'_F$. If $C\cong K_2$, then $C$ is a
star. Otherwise, choose a vertex $x\in V(C)$ with $d_C(x)\ge2$.
Every neighbor of $x$ has degree one in $C$. Since $C$ is connected,
every edge of $C$ is incident with $x$. Hence $C$ is a star.
\end{proof}

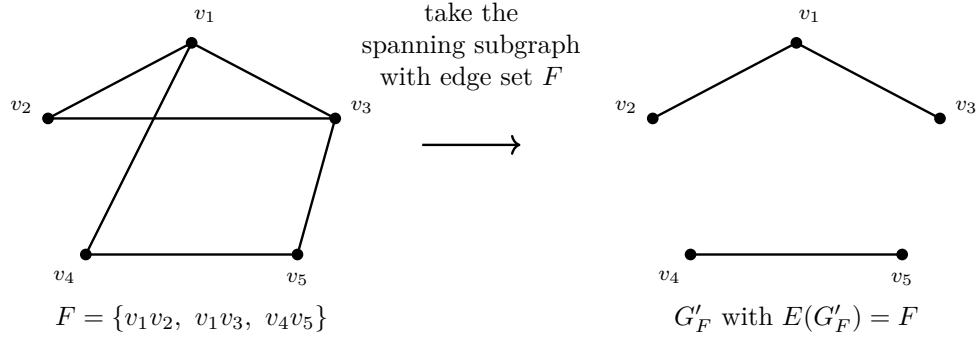
\begin{figure}[htbp]
\centering
\begin{tikzpicture}[x=1cm,y=1cm]

% ============================================================
% Left: G'
% ============================================================

%\node[font=\bfseries] at (2.7,4.05) {$G'$};

% Coordinates for G'
\coordinate (a1) at (2.7,3.0);
\coordinate (a2) at (0.8,2.0);
\coordinate (a3) at (4.6,2.0);
\coordinate (a4) at (1.3,0.2);
\coordinate (a5) at (4.1,0.2);

% Edges of G'
\draw[black,line width=0.9pt] (a1)--(a2);
\draw[black,line width=0.9pt] (a1)--(a3);
\draw[black,line width=0.9pt] (a4)--(a5);
\draw[black,line width=0.9pt] (a3)--(a5);
\draw[black,line width=0.9pt] (a1)--(a4);
\draw[black,line width=0.9pt] (a2)--(a3);

% Vertices of G'
\foreach \p in {a1,a2,a3,a4,a5}{
  \fill[black] (\p) circle (2.2pt);
}

% Labels of G'
\node[font=\scriptsize,fill=white,inner sep=1pt] at (2.88,3.35) {$v_1$};
\node[font=\scriptsize,fill=white,inner sep=1pt] at (0.45,2.15) {$v_2$};
\node[font=\scriptsize,fill=white,inner sep=1pt] at (4.95,2.15) {$v_3$};
\node[font=\scriptsize,fill=white,inner sep=1pt] at (1.02,-0.10) {$v_4$};
\node[font=\scriptsize,fill=white,inner sep=1pt] at (4.10,-0.12) {$v_5$};

% Chosen edge cover label
\node[font=\small] at (2.7,-0.65)
{$F=\{v_1v_2,\ v_1v_3,\ v_4v_5\}$};

% Arrow
\draw[->,black,line width=0.9pt] (5.75,1.65)--(7.05,1.65);

\node[font=\small,fill=white,inner sep=1pt,align=center] at (6.40,2.95)
{take the\\
spanning subgraph\\
with edge set \(F\)};

% ============================================================
% Right: G'_F
% ============================================================

\begin{scope}[shift={(8.0,0)}]

%\node[font=\bfseries] at (2.7,4.05) {$G'_F$};

% Coordinates for G'_F
\coordinate (b1) at (2.7,3.0);
\coordinate (b2) at (0.8,2.0);
\coordinate (b3) at (4.6,2.0);
\coordinate (b4) at (1.3,0.2);
\coordinate (b5) at (4.1,0.2);

% Edges of G'_F
\draw[black,line width=0.9pt] (b1)--(b2);
\draw[black,line width=0.9pt] (b1)--(b3);
\draw[black,line width=0.9pt] (b4)--(b5);

% Vertices of G'_F
\foreach \p in {b1,b2,b3,b4,b5}{
  \fill[black] (\p) circle (2.2pt);
}

% Labels of G'_F
\node[font=\scriptsize,fill=white,inner sep=1pt] at (2.88,3.35) {$v_1$};
\node[font=\scriptsize,fill=white,inner sep=1pt] at (0.45,2.15) {$v_2$};
\node[font=\scriptsize,fill=white,inner sep=1pt] at (4.95,2.15) {$v_3$};
\node[font=\scriptsize,fill=white,inner sep=1pt] at (1.02,-0.10) {$v_4$};
\node[font=\scriptsize,fill=white,inner sep=1pt] at (4.10,-0.12) {$v_5$};

\node[font=\small] at (2.7,-0.65)
{ \(G'_F\) with \(E(G'_F)=F\)};

\end{scope}

\end{tikzpicture}

\caption{The graph \(G'\) and the spanning submultigraph \(G'_F\) determined by the minimal edge cover \(F=\{v_1v_2,v_1v_3,v_4v_5\}\).}
\label{fig:Gprime-and-GprimeF}
\end{figure}

For each component of $G'_F$, choose a center $z$; if the component is a single edge, choose either endpoint. Let $z_1,\ldots,z_t$ be its leaves. Since $\Delta(G')\le3$, $1\le t\le3.$ For each $i\in[t]$, let $P_i=x^i_0x^i_1\cdots x^i_{\ell_i}$
be the $B$-thread corresponding to the edge $zz_i\in F$, oriented so
that $x^i_0=z, x^i_{\ell_i}=z_i.$
Write $e^i_j=x^i_{j-1}x^i_j$ for $j\in[\ell_i]$. Since
$F\subseteq E(G')=E(G^*)\setminus M$
and $M$ consists precisely of the edges corresponding to threads of
length one,
\begin{equation}\label{eq:selected-long}
  \ell_i\ge2
  \qquad\text{for every }i\in[t].
\end{equation}

For $i\in[3]$, define its cyclic successor by
$1^+=2, 2^+=3, 3^+=1.$
Color the edges of $P_i$, in order from $z$ to $z_i$, by
\begin{equation}\label{eq:periodic-pattern1}
  \col(e^i_j)=
  \begin{cases}
    i,   & j\equiv1\pmod3,\\
    4,   & j\equiv2\pmod3,\\
    i^+, & j\equiv0\pmod3.
  \end{cases}
\end{equation}
Thus the color sequence on $P_i$ is
$i,4,i^+,i,4,i^+,\ldots.$
Any three consecutive colors in this sequence are pairwise distinct.
Moreover, the first edges of the $t$ threads incident with $z$ receive
the distinct colors $1,\ldots,t$. Hence
\eqref{eq:periodic-pattern1} defines a proper partial edge-coloring of
all selected threads, using only the colors $1,2,3,4$; see
Figure~\ref{fig4}.

\begin{figure}[htbp]
\centering
\begin{tikzpicture}[x=0.78cm,y=0.78cm]

% Fix the bounding box so that the figure, together with the right-side list, is centered
\path[use as bounding box] (-0.25,-1.85) rectangle (9.90,4.85);

% ============================================================
% Coordinates
% ============================================================

\coordinate (v1)  at (3.0,3.35);
\coordinate (v2)  at (0.8,2.0);
\coordinate (v3)  at (5.2,2.0);
\coordinate (v4)  at (1.0,0.0);
\coordinate (v5)  at (5.0,0.0);

\coordinate (v6)  at (1.90,2.70);
\coordinate (v7)  at (4.10,2.70);
\coordinate (v8)  at (3.00,0.00);
\coordinate (v9)  at (5.10,1.00);
\coordinate (v10) at (2.35,2.30);
\coordinate (v11) at (1.65,1.15);
\coordinate (v12) at (1.35,3.75);
\coordinate (v13) at (3.00,4.15);
\coordinate (v14) at (4.65,3.75);
\coordinate (v15) at (6.05,0.00);

% ============================================================
% Uncolored edges, drawn first in light gray
% ============================================================

\draw[black!35,line width=0.8pt] (v2)--(v4);
\draw[black!35,line width=0.8pt] (v3)--(v9)--(v5);
\draw[black!35,line width=0.8pt] (v1)--(v10)--(v11)--(v4);
\draw[black!35,line width=0.8pt] (v2)--(v12)--(v13)--(v14)--(v3);
\draw[black!35,line width=0.8pt] (v5)--(v15);

% ============================================================
% Colored selected B-threads
% Red   = color 1
% Green = color 2
% Blue  = color 4
% ============================================================

% Thread v_1 v_6 v_2: color sequence 1,4
\draw[red,line width=1.45pt] (v1)--(v6);
\draw[blue,line width=1.45pt] (v6)--(v2);

% Thread v_1 v_7 v_3: color sequence 2,4
\draw[green!60!black,line width=1.45pt] (v1)--(v7);
\draw[blue,line width=1.45pt] (v7)--(v3);

% Thread v_4 v_8 v_5: color sequence 1,4
\draw[red,line width=1.45pt] (v4)--(v8);
\draw[blue,line width=1.45pt] (v8)--(v5);

% ============================================================
% Vertices
% Filled black vertices are B-vertices
% Open vertices are outside B
% ============================================================

\fill[black] (v1) circle (2.2pt);
\fill[black] (v2) circle (2.2pt);
\fill[black] (v3) circle (2.2pt);
\fill[black] (v4) circle (2.2pt);
\fill[black] (v5) circle (2.2pt);

\draw[black,fill=white,line width=0.8pt] (v6) circle (2.2pt);
\draw[black,fill=white,line width=0.8pt] (v7) circle (2.2pt);
\draw[black,fill=white,line width=0.8pt] (v8) circle (2.2pt);
\draw[black,fill=white,line width=0.8pt] (v9) circle (2.2pt);
\draw[black,fill=white,line width=0.8pt] (v10) circle (2.2pt);
\draw[black,fill=white,line width=0.8pt] (v11) circle (2.2pt);
\draw[black,fill=white,line width=0.8pt] (v12) circle (2.2pt);
\draw[black,fill=white,line width=0.8pt] (v13) circle (2.2pt);
\draw[black,fill=white,line width=0.8pt] (v14) circle (2.2pt);
\draw[black,fill=white,line width=0.8pt] (v15) circle (2.2pt);

% ============================================================
% Vertex labels, placed away from edges
% ============================================================

\node[font=\scriptsize,fill=white,inner sep=1pt] at (3.00,3.65) {$v_1$};
\node[font=\scriptsize,fill=white,inner sep=1pt] at (0.28,2.22) {$v_2$};
\node[font=\scriptsize,fill=white,inner sep=1pt] at (5.72,2.22) {$v_3$};
\node[font=\scriptsize,fill=white,inner sep=1pt] at (0.50,-0.34) {$v_4$};
\node[font=\scriptsize,fill=white,inner sep=1pt] at (5.16,-0.42) {$v_5$};

\node[font=\scriptsize,fill=white,inner sep=1pt] at (1.68,3.08) {$v_6$};
\node[font=\scriptsize,fill=white,inner sep=1pt] at (4.32,3.08) {$v_7$};
\node[font=\scriptsize,fill=white,inner sep=1pt] at (3.00,-0.42) {$v_8$};
\node[font=\scriptsize,fill=white,inner sep=1pt] at (5.55,1.05) {$v_9$};
\node[font=\scriptsize,fill=white,inner sep=1pt] at (2.82,2.22) {$v_{10}$};
\node[font=\scriptsize,fill=white,inner sep=1pt] at (2.10,1.08) {$v_{11}$};
\node[font=\scriptsize,fill=white,inner sep=1pt] at (1.08,4.12) {$v_{12}$};
\node[font=\scriptsize,fill=white,inner sep=1pt] at (3.00,4.60) {$v_{13}$};
\node[font=\scriptsize,fill=white,inner sep=1pt] at (4.92,4.12) {$v_{14}$};
\node[font=\scriptsize,fill=white,inner sep=1pt] at (6.52,-0.06) {$v_{15}$};

% ============================================================
% Edge color labels, shifted away from the edges
% ============================================================

\node[font=\scriptsize,text=red,fill=white,inner sep=1pt] at (2.30,3.28) {$1$};
\node[font=\scriptsize,text=blue,fill=white,inner sep=1pt] at (1.25,2.62) {$4$};

\node[font=\scriptsize,text=green!60!black,fill=white,inner sep=1pt] at (3.70,3.28) {$2$};
\node[font=\scriptsize,text=blue,fill=white,inner sep=1pt] at (4.80,2.62) {$4$};

\node[font=\scriptsize,text=red,fill=white,inner sep=1pt] at (2.00,0.30) {$1$};
\node[font=\scriptsize,text=blue,fill=white,inner sep=1pt] at (4.00,0.30) {$4$};

% ============================================================
% Right-side list of available color sets
% ============================================================

\draw[black,line width=0.7pt,rounded corners=2pt] (6.95,0.15) rectangle (10.5,3.5);

\node[font=\scriptsize\bfseries,anchor=west] at (7.15,3.15)
{Available sets};

\node[font=\scriptsize,anchor=west] at (7.15, 2.65)
{$\Lambda(v_1)=\{3,5\}$};

\node[font=\scriptsize,anchor=west] at (7.15, 2.15)
{$\Lambda(v_2)=\{2,3,5\}$};

\node[font=\scriptsize,anchor=west] at (7.15,1.65)
{$\Lambda(v_3)=\{1,3,5\}$};

\node[font=\scriptsize,anchor=west] at (7.15,1.15)
{$\Lambda(v_4)=\{2,3,5\}$};

\node[font=\scriptsize,anchor=west] at (7.15, 0.65)
{$\Lambda(v_5)=\{2,3,5\}$};

% ============================================================
% Legend below the graph
% ============================================================

% First row of the legend
% ============================================================
% Centered legend
% ============================================================
\begin{scope}[shift={(3.5,0)}]

% First row
\draw[red,line width=1.45pt] (-2.75,-1.10)--(-2.45,-1.10);
\node[font=\scriptsize,anchor=west] at (-2.30,-1.10) {color \(1\)};

\draw[green!60!black,line width=1.45pt] (-0.75,-1.10)--(-0.45,-1.10);
\node[font=\scriptsize,anchor=west] at (-0.30,-1.10) {color \(2\)};

\draw[blue,line width=1.45pt] (1.25,-1.10)--(1.55,-1.10);
\node[font=\scriptsize,anchor=west] at (1.70,-1.10) {color \(4\)};

\draw[black!35,line width=0.8pt] (3.40,-1.10)--(3.70,-1.10);
\node[font=\scriptsize,anchor=west] at (3.9,-1.10) {uncolored at this step};

\end{scope}

\end{tikzpicture}

\caption{The coloring of the selected \(B\)-threads in the original graph, together with the corresponding available color sets \(\Lambda(v)\) for \(v\in B\).}\label{fig4}
\end{figure}
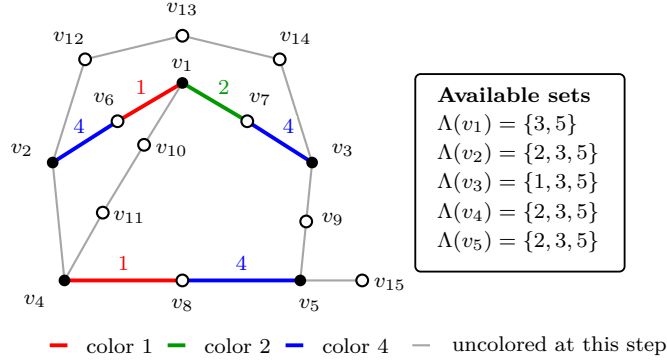

We next specify the colors available at each vertex of \(B\) for its
uncolored incident edges. Consider a star component of \(G'_F\) with
center \(z\) and leaves \(z_1,\ldots,z_t\), where \(1\le t\le3\).
For each \(i\in[t]\), the selected thread \(P_i\) starts at \(z\) with
color \(i\). Since \(\ell_i\ge2\), its second edge has color \(4\).
To ensure that every first edge \(e^i_1\) is rich in the completed
coloring, the remaining edges incident with \(z\) must avoid the
colors \(1,\ldots,t\) and \(4\). Define
\[
\Lambda(z)
=[5]\setminus\bigl(\{1,\ldots,t\}\cup\{4\}\bigr).
\]

At a leaf \(z_i\), let
\(
a_i=\col(e^i_{\ell_i-1})
\)
\(
b_i=\col(e^i_{\ell_i}).
\)
To preserve properness at \(z_i\) and ensure that
\(e^i_{\ell_i}\) is rich, the remaining edges incident with \(z_i\)
must avoid both \(a_i\) and \(b_i\). Accordingly, define
\(
\Lambda(z_i)=[5]\setminus\{a_i,b_i\}.
\)
The colors \(a_i\) and \(b_i\) are distinct by
\eqref{eq:periodic-pattern1}.

Write \(d_F(v)=d_{G'_F}(v)\). Each \(v\in B\) is incident with exactly
\(d_F(v)\) selected threads and hence with exactly \(d_F(v)\) already
colored edges. Since \(d_G(v)=3\), exactly \(3-d_F(v)\) edges incident
with \(v\) remain uncolored. Moreover, the definitions above and
\eqref{eq:periodic-pattern1} give
\begin{equation}\label{eq:port-data}
|\Lambda(v)|=4-d_F(v)
\quad\text{and}\quad
5\in\Lambda(v)
\qquad\text{for every }v\in B.
\end{equation}

\iffalse
\begin{claim}\label{cl:port-lists}
For every \(v\in B\), exactly \(3-d_F(v)\) edges incident with \(v\) remain
uncolored. Moreover,
\(
|\Lambda(v)|=4-d_F(v),
\)
and
\(
5\in\Lambda(v).
\)
\end{claim}

\begin{proof}
The selected \(F\)-threads incident with \(v\) account for precisely
\(d_F(v)\) of the three edges incident with \(v\) in \(G\). These incident
edges have already been colored when the selected threads were colored.
Therefore exactly
\(
3-d_F(v)
\)
edges incident with \(v\) remain uncolored.

It remains to compute the size of \(\Lambda(v)\). If \(v\) is a center and
\(d_F(v)=t\), then, by definition,
\(
\Lambda(v)=[5]\setminus\bigl(\{1,\ldots,t\}\cup\{4\}\bigr).
\)
Since \(t\le 3\), the color \(4\) is not contained in \(\{1,\ldots,t\}\). Hence
\(
|\Lambda(v)|=5-(t+1)=4-t=4-d_F(v).
\)
If \(v\) is a leaf of \(G'_F\), then \(d_F(v)=1\), and \(\Lambda(v)\) is obtained
from \([5]\) by deleting two distinct colors. Thus
\(
|\Lambda(v)|=3=4-d_F(v).
\)

Finally, the colors excluded from \(\Lambda(v)\) are among
\(\{1,2,3,4\}\). Hence the color \(5\) is never excluded, and so
\(
5\in\Lambda(v)
\)
for every \(v\in B\).
\end{proof}
\fi

\begin{lemma}\label{lem:selected-rich}
Suppose that all remaining uncolored edges incident with each
\(v\in B\) are assigned pairwise distinct colors from \(\Lambda(v)\).
Then the resulting partial coloring is proper at every vertex of
\(B\), and the colors on \(N_G(e)\) are pairwise distinct for every
edge \(e\) of a selected thread. Consequently, every selected-thread
edge is rich in every proper edge-coloring of \(G\) extending this
partial coloring.
\end{lemma}

\begin{proof}
First consider the center \(z\) of a star component of \(G'_F\).
The selected first edges incident with \(z\) have the distinct colors
\(1,\ldots,t\), whereas the newly colored edges incident with \(z\)
receive pairwise distinct colors from
\(
\Lambda(z)
=[5]\setminus\bigl(\{1,\ldots,t\}\cup\{4\}\bigr).
\)
Thus the coloring is proper at \(z\).

For each \(i\in[t]\), the neighbors of \(e^i_1\) are \(e^i_2\), which
has color \(4\), and the other edges incident with \(z\). Their colors
belong to the pairwise disjoint sets
\(
\{4\},
\{1,\ldots,t\}\setminus\{i\},
\Lambda(z),
\)
and are pairwise distinct within each set. Hence all neighbors of
\(e^i_1\) have distinct colors.

Now consider a leaf \(z_i\). The two other edges incident with \(z_i\)
receive distinct colors from
\(
\Lambda(z_i)=[5]\setminus\{a_i,b_i\},
\)
where
\(
a_i=\col(e^i_{\ell_i-1})
\)
and
\(
b_i=\col(e^i_{\ell_i}).
\)
Since these colors avoid \(b_i\), the coloring is proper at \(z_i\).
Since they also avoid \(a_i\), all neighbors of
\(e^i_{\ell_i}\) have pairwise distinct colors.

Finally, let \(e^i_j\) be an internal edge of \(P_i\), where
\(2\le j\le\ell_i-1\). Its endpoints lie in \(\hat G-B\) and hence
are \(2\)-vertices of \(G\). Therefore
\(
N_G(e^i_j)=\{e^i_{j-1},e^i_{j+1}\}.
\)
These two edges have distinct colors by
\eqref{eq:periodic-pattern1}.

Thus every neighbor of a selected-thread edge has already been colored,
and these colors are pairwise distinct. They are unchanged by any
further extension, so every selected-thread edge is rich in every
proper extension.
\end{proof}

\color{black}

It remains to color the uncolored edges incident with \(B\). We encode these edges in an auxiliary graph \(Q\). Let \(X_2\) be the set of internal vertices of the unselected \(B\)-threads of length two, and
let \(X_p\) be the set of leaves of \(G\) adjacent to \(B\). For each unselected \(B\)-thread \(P\) of length at least three, introduce two new vertices \(x_P\) and \(y_P\), with all such vertices distinct and outside \(V(G)\), and let \(X_{\ge3}\) be their set. Let \(Q\) be the graph with vertex set
\[
V(Q)=B\cup X_2\cup X_{\ge3}\cup X_p
\]
and edge set defined as follows:
\begin{enumerate}[label=\textup{(\roman*)}]
\item for each unselected \(B\)-thread \(P=uv\) of length one, add
the edge \(uv\);
\item for each unselected \(B\)-thread \(P=uxv\) of length two, add
the edges \(ux\) and \(xv\);
\item for each unselected \(B\)-thread \(P\) of length at least three
with endpoints \(u,v\), add the disjoint pendant edges \(ux_P\) and
\(vy_P\), representing the two end edges of \(P\);
\item for each pendant edge \(vu\) of \(G\) with \(v\in B\), add the
edge \(vu\).
\end{enumerate}

Every edge of \(Q\) has one or two endpoints in \(B\). For each
\(v\in B\), its incident edges in \(Q\) correspond bijectively to the
uncolored edges of \(G\) incident with \(v\). Hence
\(
d_Q(v)=3-d_F(v)\le2,
\)
where the inequality follows because \(F\) is an edge cover. Moreover,
vertices in \(X_2\) have degree two, whereas those in
\(X_{\ge3}\cup X_p\) have degree one. Thus
\(
\Delta(Q)\le2.
\)

We assign a list \(L(e)\) to each edge \(e\in E(Q)\). If \(v\) is the
unique endpoint of \(e\) in \(B\), set
\(
L(e)=\Lambda(v).
\)
If \(e=uv\) has both endpoints in \(B\), then \(e\) corresponds to an
unselected \(B\)-thread of length one; in this case, set
\begin{equation}\label{eq:direct-list}
L(e)=\Lambda(u)\cap\Lambda(v).
\end{equation}
By~\eqref{eq:port-data}, every list \(L(e)\) contains \(5\) and is
therefore nonempty.

Figure~\ref{fig:auxiliary-Q-lists} illustrates the construction of
\(Q\) and its edge lists. Vertices of \(B\) are shown as filled black
circles, other vertices inherited from \(G\) as open circles, and new
vertices as squares.

\color{black}

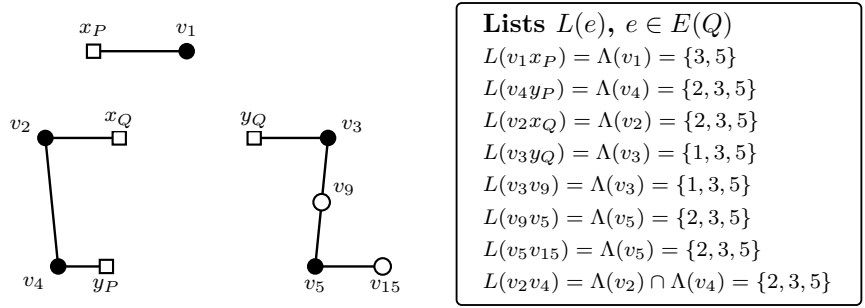
\begin{figure}[htbp]
\centering
\begin{tikzpicture}[x=0.85cm,y=0.85cm]

% ============================================================
% Styles
% ============================================================
\tikzset{
  Bv/.style={circle,fill=black,inner sep=2.2pt},
  Ov/.style={circle,draw=black,fill=white,line width=0.8pt,inner sep=2.2pt},
  Nv/.style={rectangle,draw=black,fill=white,line width=0.8pt,
             minimum width=5pt,minimum height=5pt,inner sep=0pt}
}

% Bounding box
\path[use as bounding box] (-0.6,-1.2) rectangle (15.2,4.7);

% ============================================================
% Coordinates
% Keep the original vertices in the same positions as in G
% ============================================================

% B-vertices
\coordinate (v1)  at (3.00,3.35);
\coordinate (v2)  at (0.80,2.00);
\coordinate (v3)  at (5.20,2.00);
\coordinate (v4)  at (1.00,0.00);
\coordinate (v5)  at (5.00,0.00);

% Original retained vertices
\coordinate (v9)  at (5.10,1.00);
\coordinate (v15) at (6.05,0.00);

% New vertices introduced in Q
\coordinate (xP)  at (1.55,3.35);  % for the unselected 3-thread at v1
\coordinate (yP)  at (1.75,0.00);  % for the unselected 3-thread at v4
\coordinate (xQ)  at (1.95,2.00);  % for the unselected 4-thread at v2
\coordinate (yQ)  at (4.05,2.00);  % for the unselected 4-thread at v3

% ============================================================
% Edges of Q
% ============================================================

% (i) from the unselected 1-thread
\draw[black,line width=0.9pt] (v2)--(v4);

% (ii) from the unselected 2-thread
\draw[black,line width=0.9pt] (v3)--(v9)--(v5);

% (iii) from the unselected 3-thread
\draw[black,line width=0.9pt] (v1)--(xP);
\draw[black,line width=0.9pt] (v4)--(yP);

% (iii) from the unselected 4-thread
\draw[black,line width=0.9pt] (v2)--(xQ);
\draw[black,line width=0.9pt] (v3)--(yQ);

% pendant edge of G
\draw[black,line width=0.9pt] (v5)--(v15);

% ============================================================
% Vertices
% ============================================================

% B-vertices: filled black circles
\node[Bv] at (v1) {};
\node[Bv] at (v2) {};
\node[Bv] at (v3) {};
\node[Bv] at (v4) {};
\node[Bv] at (v5) {};

% Original non-B vertices kept in Q: open circles
\node[Ov] at (v9) {};
\node[Ov] at (v15) {};

% New vertices: squares
\node[Nv] at (xP) {};
\node[Nv] at (yP) {};
\node[Nv] at (xQ) {};
\node[Nv] at (yQ) {};

% ============================================================
% Labels
% ============================================================

% Labels of B-vertices
\node[font=\scriptsize,fill=white,inner sep=1pt] at (3.00,3.68) {$v_1$};
\node[font=\scriptsize,fill=white,inner sep=1pt] at (0.42,2.18) {$v_2$};
\node[font=\scriptsize,fill=white,inner sep=1pt] at (5.58,2.18) {$v_3$};
\node[font=\scriptsize,fill=white,inner sep=1pt] at (0.62,-0.26) {$v_4$};
\node[font=\scriptsize,fill=white,inner sep=1pt] at (5.00,-0.30) {$v_5$};

% Labels of original non-B vertices
\node[font=\scriptsize,fill=white,inner sep=1pt] at (5.45,1.20) {$v_9$};
\node[font=\scriptsize,fill=white,inner sep=1pt] at (6.10,-0.30) {$v_{15}$};

% Labels of new vertices
\node[font=\scriptsize,fill=white,inner sep=1pt] at (1.55,3.68) {$x_P$};
\node[font=\scriptsize,fill=white,inner sep=1pt] at (1.75,-0.30) {$y_P$};
\node[font=\scriptsize,fill=white,inner sep=1pt] at (1.95,2.28) {$x_Q$};
\node[font=\scriptsize,fill=white,inner sep=1pt] at (4.05,2.28) {$y_Q$};

% Title under the graph
%\node[font=\small] at (3.0,-0.85) {The auxiliary graph \(Q\)};

% ============================================================
% Right-side list assignments
% ============================================================

\draw[black,line width=0.7pt,rounded corners=2pt] (7.20,-0.60) rectangle (13.5,4.10);

\node[font=\small\bfseries,anchor=west] at (7.45,3.75) {Lists \(L(e)\), \(e\in E(Q)\)};

\node[font=\scriptsize,anchor=west] at (7.45,3.25)
{$L(v_1x_P)=\Lambda(v_1)=\{3,5\}$};

\node[font=\scriptsize,anchor=west] at (7.45,2.75)
{$L(v_4y_P)=\Lambda(v_4)=\{2,3,5\}$};

\node[font=\scriptsize,anchor=west] at (7.45,2.25)
{$L(v_2x_Q)=\Lambda(v_2)=\{2,3,5\}$};

\node[font=\scriptsize,anchor=west] at (7.45,1.75)
{$L(v_3y_Q)=\Lambda(v_3)=\{1,3,5\}$};

\node[font=\scriptsize,anchor=west] at (7.45,1.25)
{$L(v_3v_9)=\Lambda(v_3)=\{1,3,5\}$};

\node[font=\scriptsize,anchor=west] at (7.45,0.75)
{$L(v_9v_5)=\Lambda(v_5)=\{2,3,5\}$};

\node[font=\scriptsize,anchor=west] at (7.45,0.25)
{$L(v_5v_{15})=\Lambda(v_5)=\{2,3,5\}$};

\node[font=\scriptsize,anchor=west] at (7.45,-0.25)
{$L(v_2v_4)=\Lambda(v_2)\cap\Lambda(v_4)=\{2,3,5\}$};

\end{tikzpicture}
\caption{The auxiliary graph \(Q\) and the list \(L(e)\) assigned to each edge \(e\in E(Q)\).}
\label{fig:auxiliary-Q-lists}
\end{figure}

\begin{lemma}\label{lem:sub2-list}
Let \(H\) be a graph with \(\Delta(H)\le2\), and let \(L\) assign a
list of colors to each edge of \(H\). Suppose that \(|L(e)|\ge2\) for
every \(e\in E(H)\), and that every cycle component of \(H\) contains
an edge \(e\) with \(|L(e)|\ge3\). Then \(H\) has a proper
edge-coloring \(\psi\) such that
\(
\psi(e)\in L(e)
\)
for every 
\(
e\in E(H).
\)
\end{lemma}

\begin{proof}
Every component of \(H\) containing an edge is a path or a cycle.
A path can be colored greedily from one end, since each edge has at
most one previously colored neighbor and a list of size at least two.

For a cycle component \(C\), choose \(e_0\in E(C)\) with
\(|L(e_0)|\ge3\). Color the path \(C-e_0\) greedily. The edge \(e_0\)
then has at most two colored neighbors, so a color in \(L(e_0)\)
distinct from their colors remains available. Coloring the components
independently completes the proof.
\end{proof}

\begin{lemma}\label{lem:port-list1}
The graph \(Q\) has a proper edge-coloring \(\psi\) such that
\(\psi(e)\in L(e)\) for every \(e\in E(Q)\).
\end{lemma}

\begin{proof}
Let
\(
S=\{e\in E(Q):|L(e)|=1\}.
\)
Suppose that \(e\) has a unique endpoint \(v\) in \(B\). By
\eqref{eq:port-data} and \(d_Q(v)=3-d_F(v)\),
\(
|L(e)|=|\Lambda(v)|
       =4-d_F(v)
       =1+d_Q(v)
       \ge2.
\)
Thus every edge of \(S\) has both endpoints in \(B\) and hence
corresponds to an unselected \(B\)-thread of length one. Since every
list contains \(5\), we have \(L(e)=\{5\}\) for each \(e\in S\).
Moreover, \(S\) is a matching by~\eqref{eq:M-matching1}.

Color every edge of \(S\) with \(5\), let \(Q'=Q-S\), and define, for
each \(e\in E(Q')\),
\[
L'(e)=
\begin{cases}
L(e)\setminus\{5\},
   & \text{if \(e\) is adjacent to an edge of \(S\)},\\
L(e),
   & \text{otherwise}.
\end{cases}
\]

We first show that
\(
|L'(e)|\ge2
\)
for every 
\(
e\in E(Q').
\)
If both endpoints of \(e\) lie in \(B\), then \(e\) corresponds to a length-one thread. Since all such edges form a matching, \(e\) is not adjacent to \(S\). Hence \(L'(e)=L(e)\), and \(|L'(e)|\ge2\) because
\(e\notin S\). Now suppose that \(e\) has a unique endpoint \(v\) in \(B\). If no edge of \(S\) is incident with \(v\), then \(L'(e)=L(e)\), and the equality above gives \(|L'(e)|\ge2\). Otherwise, \(e\) and an edge of
\(S\) are both incident with \(v\). Since \(d_Q(v)\le2\), we have \(d_Q(v)=2\), and therefore
\(
|L(e)|=1+d_Q(v)=3.
\)
Deleting \(5\) leaves \(|L'(e)|=2\).

It remains to verify the condition on cycle components. Let \(C\) be
a cycle component of \(Q'\). No edge arising from a pendant edge or
from a thread of length at least three can lie on \(C\), since such an
edge has a degree-one endpoint in \(Q\). Thus every edge of \(C\)
either corresponds to a length-one thread or belongs to a retained
length-two thread. Since the length-one edges form a matching, \(C\)
contains an edge \(e\) from a length-two thread.

Let \(v\in B\) be the endpoint of \(e\) in \(B\). Since \(v\) lies on
the cycle \(C\),
\(
d_{Q'}(v)=2.
\)
As \(d_Q(v)\le2\), it follows that \(d_Q(v)=2\), and no edge of \(S\)
is incident with \(v\). The other endpoint of \(e\) lies outside
\(B\), whereas every edge of \(S\) has both endpoints in \(B\).
Hence \(e\) is not adjacent to \(S\), and
\(
|L'(e)|=|L(e)|=1+d_Q(v)=3.
\)
Therefore every cycle component of \(Q'\) contains an edge whose list
has size at least three.

Since \(\Delta(Q')\le2\), Lemma~\ref{lem:sub2-list} gives a proper
\(L'\)-edge-coloring of \(Q'\). Together with the edges of \(S\), all
colored \(5\), this yields a proper edge-coloring \(\psi\) of \(Q\):
the edges of \(S\) form a matching, and \(5\) was deleted from the
lists of all edges adjacent to \(S\). Finally,
\(L'(e)\subseteq L(e)\) for every \(e\in E(Q')\), so
\(\psi(e)\in L(e)\) for every \(e\in E(Q)\).
\end{proof}

Transfer \(\psi(e)\) to the edge of \(G\) represented by each
\(e\in E(Q)\). For an unselected thread of length at least three, this
colors only its two end edges. Since \(\psi\) is proper and respects
\(L\), the previously uncolored edges at each \(v\in B\) now receive
pairwise distinct colors from \(\Lambda(v)\).
Lemma~\ref{lem:selected-rich} then shows that the current coloring is
proper at \(B\) and that the colors on \(N_G(f)\) are pairwise distinct
for every selected-thread edge \(f\). This transfer is illustrated in
Figure~\ref{fig:transfer-Q-to-G}.

It remains to color the internal edges of the unselected threads of
length at least three. Fix such a thread \(P\), with edges
\(e_1,\ldots,e_\ell\) ordered from \(u\in B\) to \(v\in B\), and put
\(a=\col(e_1)\) and \(b=\col(e_\ell)\).

Suppose first that \(\ell=3\). Let \(C(u)\) be the set of colors on
the three edges incident with \(u\). Since \(|C(u)|=3\), we may choose
\begin{equation}\label{eq:length-three}
\col(e_2)\in[5]\setminus\bigl(C(u)\cup\{b\}\bigr).
\end{equation}
Indeed, at most four colors are excluded. Since \(a\in C(u)\), this
choice extends the coloring properly to \(P\). Moreover, the other two
edges incident with \(u\) have distinct colors in \(C(u)\), whereas
\(\col(e_2)\notin C(u)\). Hence the colors on \(N_G(e_1)\) are
pairwise distinct.

If \(\ell\ge4\), Lemma~\ref{lem:path-sequence} extends the prescribed
colors \(a\) and \(b\) to a strong \(5\)-edge-coloring of \(P\).
Since the internal vertices of \(P\) are \(2\)-vertices of \(G\), the
colors on the neighborhood of every internal edge of \(P\) are
pairwise distinct.

Applying these extensions independently to all such threads yields a
proper \(5\)-edge-coloring of \(G\). In this coloring, every
selected-thread edge, the first edge of every unselected thread of
length three, and every internal edge of every unselected thread of
length at least four is rich.

\begin{figure}[htbp]
\centering
\begin{tikzpicture}[x=0.85cm,y=0.85cm]

% ============================================================
% Styles
% ============================================================
\tikzset{
  Bv/.style={circle,fill=black,inner sep=2.2pt},
  Ov/.style={circle,draw=black,fill=white,line width=0.8pt,inner sep=2.2pt},
  Nv/.style={rectangle,draw=black,fill=white,line width=0.8pt,
             minimum width=5pt,minimum height=5pt,inner sep=0pt},
  c1/.style={red,line width=0.95pt},
  c2/.style={green!60!black,line width=0.95pt},
  c3/.style={yellow!70!orange,line width=0.95pt},
  c4/.style={blue,line width=0.95pt},
  c5/.style={purple,line width=0.95pt},
  lab/.style={font=\scriptsize,fill=white,inner sep=1pt}
}

% ============================================================
% Bounding box
% ============================================================
\path[use as bounding box] (-0.7,-1.55) rectangle (16.0,5.0);

% ============================================================
% Left graph: auxiliary graph Q with coloring psi
% ============================================================
\begin{scope}[shift={(0,0)}]

% Coordinates of Q
\coordinate (v1)  at (3.00,3.35);
\coordinate (v2)  at (0.80,2.00);
\coordinate (v3)  at (5.20,2.00);
\coordinate (v4)  at (1.00,0.00);
\coordinate (v5)  at (5.00,0.00);

\coordinate (v9)  at (5.10,1.00);
\coordinate (v15) at (6.05,0.00);

\coordinate (xP)  at (1.55,3.35);
\coordinate (yP)  at (1.75,0.00);
\coordinate (xQ)  at (1.95,2.00);
\coordinate (yQ)  at (4.05,2.00);

% Edges of Q with colors psi
\draw[c5] (v2)--(v4);
\node[lab] at (0.48,0.98) {$5$};

\draw[c3] (v3)--(v9);
\node[lab] at (5.60,1.58) {$3$};

\draw[c2] (v9)--(v5);
\node[lab] at (5.58,0.42) {$2$};

\draw[c3] (v1)--(xP);
\node[lab] at (2.28,3.65) {$3$};

\draw[c2] (v4)--(yP);
\node[lab] at (1.40,-0.28) {$2$};

\draw[c2] (v2)--(xQ);
\node[lab] at (1.38,2.30) {$2$};

\draw[c1] (v3)--(yQ);
\node[lab] at (4.62,2.30) {$1$};

\draw[c3] (v5)--(v15);
\node[lab] at (5.56,-0.28) {$3$};

% Vertices
\node[Bv] at (v1) {};
\node[Bv] at (v2) {};
\node[Bv] at (v3) {};
\node[Bv] at (v4) {};
\node[Bv] at (v5) {};

\node[Ov] at (v9) {};
\node[Ov] at (v15) {};

\node[Nv] at (xP) {};
\node[Nv] at (yP) {};
\node[Nv] at (xQ) {};
\node[Nv] at (yQ) {};

% Labels
\node[lab] at (3.00,3.72) {$v_1$};
\node[lab] at (0.6,2.32) {$v_2$};
\node[lab] at (5.65,2.20) {$v_3$};
\node[lab] at (0.75,-0.30) {$v_4$};
\node[lab] at (5.02,-0.34) {$v_5$};

\node[lab] at (5.48,1.08) {$v_9$};
\node[lab] at (6.22,-0.34) {$v_{15}$};

\node[lab] at (1.55,3.72) {$x_P$};
\node[lab] at (1.85,-0.30) {$y_P$};
\node[lab] at (1.95,2.32) {$x_Q$};
\node[lab] at (4.05,2.32) {$y_Q$};

%\node[font=\small] at (3.0,-0.95) {The auxiliary graph \(Q\) with \(\psi\)};

\end{scope}

% ============================================================
% Transfer arrow
% ============================================================
\draw[->,line width=1.0pt] (6.80,1.65)--(7.65,1.65);
\node[font=\small,align=center] at (7.22,2.48) {transfer\\colors};

% ============================================================
% Right graph: original graph G with transferred coloring
% ============================================================
\begin{scope}[shift={(8.65,0)}]

% Coordinates of G
\coordinate (g1)  at (3.0,3.35);
\coordinate (g2)  at (0.8,2.0);
\coordinate (g3)  at (5.2,2.0);
\coordinate (g4)  at (1.0,0.0);
\coordinate (g5)  at (5.0,0.0);

\coordinate (g6)  at (1.90,2.70);
\coordinate (g7)  at (4.10,2.70);
\coordinate (g8)  at (3.00,0.00);
\coordinate (g9)  at (5.10,1.00);
\coordinate (g10) at (2.35,2.30);
\coordinate (g11) at (1.65,1.15);
\coordinate (g12) at (1.35,3.75);
\coordinate (g13) at (3.00,4.15);
\coordinate (g14) at (4.65,3.75);
\coordinate (g15) at (6.05,0.00);

% Selected F-threads
\draw[c1] (g1)--(g6);
\node[lab] at (2.18,3.12) {$1$};

\draw[c4] (g6)--(g2);
\node[lab] at (1.26,2.54) {$4$};

\draw[c2] (g1)--(g7);
\node[lab] at (3.82,3.12) {$2$};

\draw[c4] (g7)--(g3);
\node[lab] at (4.76,2.55) {$4$};

\draw[c1] (g4)--(g8);
\node[lab] at (1.95,-0.28) {$1$};

\draw[c4] (g8)--(g5);
\node[lab] at (4.05,-0.28) {$4$};

% Unselected 1-thread
\draw[c5] (g2)--(g4);
\node[lab] at (0.48,0.98) {$5$};

% Unselected 2-thread
\draw[c3] (g3)--(g9);
\node[lab] at (5.38,1.58) {$3$};

\draw[c2] (g9)--(g5);
\node[lab] at (5.28, 0.5) {$2$};

% Unselected 3-thread
\draw[c3] (g1)--(g10);
\node[lab] at (2.9,2.7) {$3$};

\draw[c4] (g10)--(g11);
\node[lab] at (2.25,1.74) {$4$};

\draw[c2] (g11)--(g4);
\node[lab] at (1.66,0.60) {$2$};

% Unselected 4-thread
\draw[c2] (g2)--(g12);
\node[lab] at (0.86,3.08) {$2$};

\draw[c3] (g12)--(g13);
\node[lab] at (2.12,4.15) {$3$};

\draw[c4] (g13)--(g14);
\node[lab] at (3.88,4.15) {$4$};

\draw[c1] (g14)--(g3);
\node[lab] at (5.15,3.08) {$1$};

% Pendant edge
\draw[c3] (g5)--(g15);
\node[lab] at (5.58,-0.28) {$3$};

% Vertices
\foreach \p in {g1,g2,g3,g4,g5}{
  \node[Bv] at (\p) {};
}
\foreach \p in {g6,g7,g8,g9,g10,g11,g12,g13,g14,g15}{
  \node[Ov] at (\p) {};
}

% Labels of G
\node[lab] at (3.00,3.72) {$v_1$};
\node[lab] at (0.28,2.18) {$v_2$};
\node[lab] at (5.72,2.18) {$v_3$};
\node[lab] at (0.50,-0.32) {$v_4$};
\node[lab] at (5.14,-0.40) {$v_5$};

\node[lab] at (1.62,3.06) {$v_6$};
\node[lab] at (4.38,3.0) {$v_7$};
\node[lab] at (3.00,-0.40) {$v_8$};
\node[lab] at (5.55,1.08) {$v_9$};
\node[lab] at (2.78,2.12) {$v_{10}$};
\node[lab] at (2.18,1.10) {$v_{11}$};
\node[lab] at (1.08,4.12) {$v_{12}$};
\node[lab] at (3.00,4.58) {$v_{13}$};
\node[lab] at (4.92,4.12) {$v_{14}$};
\node[lab] at (6.52,-0.05) {$v_{15}$};

%\node[font=\small] at (3.05,-0.98) {The original graph \(G\) with the transferred coloring};

\end{scope}

% ============================================================
% Color legend
% ============================================================
\begin{scope}[shift={(1.95,0)}]
\draw[red,line width=1.45pt] (0.75,-1.25)--(1.05,-1.25);
\node[font=\scriptsize,anchor=west] at (1.20,-1.25) {color \(1\)};

\draw[green!60!black,line width=1.45pt] (2.75,-1.25)--(3.05,-1.25);
\node[font=\scriptsize,anchor=west] at (3.20,-1.25) {color \(2\)};

\draw[yellow!70!orange,line width=1.45pt] (4.75,-1.25)--(5.05,-1.25);
\node[font=\scriptsize,anchor=west] at (5.20,-1.25) {color \(3\)};

\draw[blue,line width=1.45pt] (6.75,-1.25)--(7.05,-1.25);
\node[font=\scriptsize,anchor=west] at (7.20,-1.25) {color \(4\)};

\draw[purple,line width=1.45pt] (8.75,-1.25)--(9.05,-1.25);
\node[font=\scriptsize,anchor=west] at (9.20,-1.25) {color \(5\)};
\end{scope}

\end{tikzpicture}
\caption{The chosen colors on \(Q\) and the resulting coloring of \(G\).}
\label{fig:transfer-Q-to-G}
\end{figure}
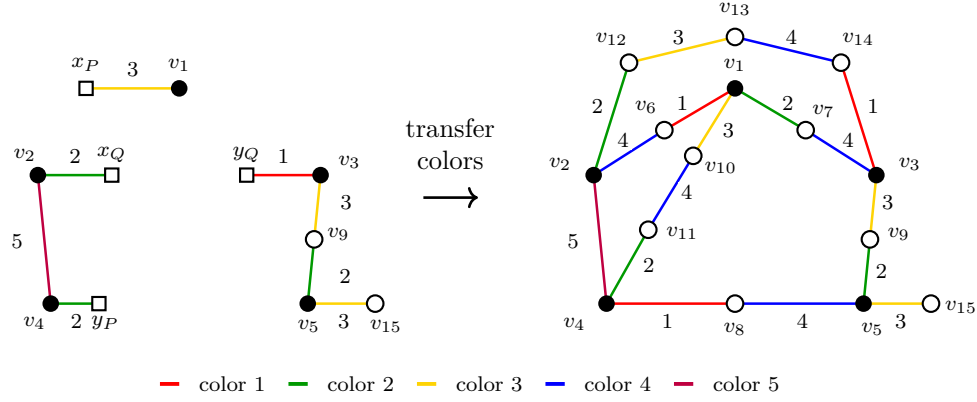

\begin{proof}[Completion of the proof of Theorem~\ref{thm:main}]
The preceding construction yields a proper \(5\)-edge-coloring
\(\col\) of \(G\). Properness at vertices of \(B\) follows from
Lemmas~\ref{lem:port-list1} and~\ref{lem:selected-rich}. At the
internal vertices of selected threads it follows from the periodic
pattern~\eqref{eq:periodic-pattern1}; for unselected threads it follows
from the proper coloring of \(Q\) when the length is two, from
\eqref{eq:length-three} when the length is three, and from
Lemma~\ref{lem:path-sequence} when the length is at least four.
Properness at leaves is immediate.

By Lemma~\ref{lem:selected-rich}, every selected-thread edge is rich.
Since each selected thread has length at least two by
\eqref{eq:selected-long}, every such edge is adjacent to another rich
edge on the same thread.

Because \(F\) covers \(B\), every vertex of \(B\) is incident with a
selected thread. Hence every pendant edge and every edge of an
unselected thread incident with \(B\) has a rich selected-thread
neighbor. This covers all edges of unselected threads of length at
most two and the end edges of all longer unselected threads.

For an unselected thread of length three, its middle edge is adjacent
to the first end edge, which is rich by the choice in
\eqref{eq:length-three}. For an unselected thread of length at least
four, all internal edges are rich, and there are at least two
consecutive internal edges. Thus every internal edge has a rich
internal neighbor.

Consequently, every non-isolated edge of \(G\) has a rich neighbor.
Therefore \(\col\) is a rich-neighbor \(5\)-edge-coloring of \(G\),
contradicting the choice of \(G\).
\end{proof}

\section{Concluding remarks}

We proved that every subcubic graph \(G\) satisfying
\(\Delta(G[V_3])\le1\) admits a rich-neighbor \(5\)-edge-coloring.
Thus the conjecture of Petruševski and Škrekovski holds whenever each
\(3\)-vertex has at most one \(3\)-neighbor, extending the previously
known case in which the \(3\)-vertices are independent. The bound is
best possible, as shown by the diamond \(K_4-e\).

The proof combines a thread decomposition with an minimal
edge cover and a list edge-coloring of an auxiliary graph of maximum
degree two. A natural next step is to consider graphs satisfying
\(\Delta(G[V_3])\le2\). In that setting, the length-one threads need
not form a matching, so the matching argument for the singleton lists
in \(Q\) no longer applies directly. Overcoming this compatibility
obstruction may provide further progress toward the full conjecture.

\section*{Acknowledgments}

This work was supported by the National Natural Science Foundation of China (No. 12271235) and Fujian Alliance of Mathematics (No. 2023SXLMMS09).

\section*{Data Availability Statement}

No datasets were generated or analyzed during the current study. 

\section*{Conflict of Interest}

The authors declare that they have no conflict of interest.

\end{document}